\documentclass[12pt,a4paper]{article}
\usepackage{amsmath}
\usepackage{amssymb}
\usepackage{amsthm}

\usepackage{verbatim}

\usepackage[utf8]{inputenc}

\usepackage{graphicx}
\usepackage{color}
\usepackage{enumerate}

\usepackage{esint}
\usepackage[all]{xy}
\usepackage{framed}
\usepackage{bbm}
\usepackage{subcaption}

\usepackage{bbm}

\theoremstyle{plain}
\newtheorem{theorem}{Theorem}
\newtheorem{corollary}[theorem]{Corollary}
\newtheorem{lemma}{Lemma}[section]
\newtheorem{claim}{Claim}[section]
\newtheorem{proposition}{Proposition}[section]

\theoremstyle{definition}
\newtheorem{definition}{Definition}[section]
\newtheorem{remark}{Remark}[section]
\newtheorem{example}{Example}[section]
\numberwithin{equation}{section}

\newcommand{\bfa}{\mathbf{a}}

\newcommand{\bfp}{\mathbf{p}}
\newcommand{\bfe}{\mathbf{e}}

\newcommand{\D}{\mathcal{D}}
\newcommand{\cK}{\mathcal{K}}

\newcommand{\cH}{\mathcal{H}}

\newcommand{\cE}{\mathcal{E}}

\newcommand{\dd}{\mathrm{\,d}}
\newcommand{\hx}{\hat{x}}

\newcommand{\mR}{\mathbb{R}}

\newcommand{\oR}{\overline{\mR}}
\newcommand{\oF}{\overline{F}}

\newcommand{\bben}{\mathbbm{1}}

\DeclareMathOperator{\tr}{tr}
\DeclareMathOperator{\diag}{diag}

\DeclareMathOperator{\di}{div}

\DeclareMathOperator{\dist}{dist}
\DeclareMathOperator{\con}{conv}

\DeclareMathOperator{\rot}{rot}

\DeclareMathOperator{\ac}{ac}
\DeclareMathOperator{\inter}{int}
\DeclareMathOperator{\clos}{cl}

\begin{document}

\author{Karl K. Brustad}
\title{Sobolev gradients of viscosity supersolutions}
\maketitle

\begin{abstract}
\noindent
We investigate which elliptic PDEs that have the property that every viscosity supersolution is $W^{1,q}_{loc}(\Omega)$, $\Omega\subseteq\mR^n$.
The asymptotic cone of the operator's sublevel set seems to be essential. It turns out that much can be said if we know how this cone  compares to the sublevel set of a certain minimal operator associated with the exponent $q$.
\end{abstract}

\section{Introduction}
A viscosity supersolution of an elliptic equation is, a priori, no more regular than lower semicontinuous.
Its definition does not require any differentiability.
However, some equations are known to impose a regularity on their supersolutions.
For example,
superharmonic functions have weak gradients that are locally integrable provided the exponent is sufficiently small. If $u$ is a supersolution to the Laplace equation $\Delta u = 0$ in some open set $\Omega\subseteq \mR^n$, $n\geq 2$, then
\[\int_D|\nabla u|^q\dd x < \infty\]
whenever
\[0 < q < \frac{n}{n-1}\]
for every compact subset $D$ of $\Omega$.
That is, $u$ belongs to the Sobolev space $W^{1,q}_{loc}(\Omega)$. The fundamental solution $x\mapsto |x|^{2-n}$, or $-\ln|x|$ in the case $n=2$, shows that the bound on $q$ is sharp. This result is generalized to $p$-superharmonic functions in \cite{MR3931688}. The exponent $q$ can be increased when $p>2$. More precisely, if $u$ is a supersolution to the $p$-Laplace equation
\begin{equation}
\Delta_p u := \di(|\nabla u|^{p-2}\nabla u) = 0
\label{eq:plap}
\end{equation}
for some $2\leq p\leq n$, then $u\in W^{1,q}_{loc}(\Omega)$ for every $q$ such that
\[0 < q < \frac{n}{n-1}(p-1).\]
Again the bound is sharp, as confirmed by the fundamental solution $x\mapsto |x|^\frac{p-n}{p-1}$ -- or $-\ln|x|$ in the case $p=n$.\footnote{The viscosity supersolutions of \eqref{eq:plap} coincide with the traditional $p$-superharmonic functions defined by differentiating test functions under the integral sign. See \cite{MR1871417}.}

One may wonder what it is that characterizes a PDE that has every supersolution in some first order Sobolev space.
In this paper we begin the investigation by considering equations that depend only on the second order partial derivatives
\[\cH u := \left[\frac{\partial^2 u}{\partial x_i\partial x_j}\right]_{i,j}.\]
Given an exponent $q$, we ask the following question. For which operators $F\colon S(n)\to\mR$ is every viscosity supersolution of the equation
\[F(\cH u) = 0\qquad\text{in $\Omega$}\]
in $W^{1,q}_{loc}(\Omega)$?
To our knowledge, this particular problem has not been addressed before.

Our theorems are presented in the next Section. Sufficient and necessary conditions are established in Theorem \ref{thm:suff} and Theorem \ref{thm:nec}, respectively. The \emph{dominative $p$-Laplacian}
\[\D_p u := \Delta u + (p-2)\lambda_n(\cH u)\]
will play a prominent role in the characterization. It was introduced in \cite{MR4085709} in order to explain a superposition principle in the $p$-Laplace equation \eqref{eq:plap}. The key property in that setting was its domination
\begin{equation}
|\nabla u|^{2-p}\Delta_p u \leq \D_p u
\label{eq:domination}
\end{equation}
over the normalized $p$-Laplacian.
In the present situation, a sort of opposite property will also be of importance.
It is shown in Section \ref{sec:risee} that $\D_p$ is \emph{minimal} in the class of sublinear elliptic operators that share its fundamental solution \eqref{sublin_fundsol}.
This is a result of independent interest and we consider it as one of the main contributions of this paper.
We shall also make use of a tool developed in \cite{brustad2020comparison}. Some relevant properties of the \emph{associated consistent distance operator} are proved in Section \ref{sec:acdo}.

The definition and the elementary theory of viscosity solutions can be found in \cite{MR1118699}.
A function $u\in C^2(\Omega)$ is a viscosity supersolution if and only if $F(\cH u(x))\leq 0$ for all $x\in\Omega$.

\section{The main theorems}

In order to start the search for operators having only $W^{1,q}_{loc}$ supersolutions, we make some observations. Firstly, the properties we are looking for must depend only on the sublevel set
\[\Theta = \Theta(F) := \left\{X\in S(n)\;|\; F(X)\leq 0\right\}\]
in the space $S(n)$ of symmetric $n\times n$ matrices. Indeed, if two operators have the same sublevel set, then they also share the same set of supersolutions. Secondly, the properties should be invariant under translations of $\Theta$. This is because a function $u$ is a supersolution to $F(\cH u) = 0$ if and only if $v(x) := u(x) + \frac{1}{2}x^\top X_0x$ is a supersolution to the equation $F(\cH v - X_0) = 0$. The sublevel set of $X\mapsto F(X-X_0)$ is the translation $\Theta + \{X_0\}$, and $u$ and $v$ are clearly in the same Sobolev space. Also, a linear transformation of the form
\[B^\top \Theta B := \{B^\top XB\;|\; X\in\Theta\}\]
where $B$ is an invertible $n\times n$ matrix, should not matter: If $u$ is a supersolution to $F(\cH u) = 0$, define the function $v(x) := u(Bx)$. Then $\cH v(x) = B^\top \cH u(Bx)B$ and $v$ is a supersolution to the equation $F(B^{-\top}\cH v B^{-1}) = 0$ with sublevel set $B^\top \Theta B$. Again, $u\in W^{1,q}_{loc}$ if and only if $v\in W^{1,q}_{loc}$. Admittedly, we conducted this argument as if $u$ and $v$ were twice differentiable, but, as we shall see, the reasoning is sound because we can do the computations on the test functions.
Finally, if a supersolution is not smooth, one can suspect that its Hessian matrix has to run off to infinity in some direction in $\Theta\subseteq S(n)$. It is perhaps only the shape of $\Theta$ \emph{for large} $\|X\|$ that is significant for whether a supersolution is in $W^{1,q}_{loc}$ or not. We use the \emph{asymptotic cone}
\[\ac(\Theta) := \left\{ Z\in S(n)\;\middle|\; \exists t_k\to\infty,\,\exists X_k\in\Theta\text{ with } \lim_{k\to\infty}\frac{X_k}{t_k} = Z \right\}\]
to capture the behavior of $\Theta$ at infinity.

The dominative $p$-Laplacian $\D_p\colon C^2(\Omega)\to C(\Omega)$ can be written as $\D_p u(x) = F_p(\cH u(x))$ where $F_p\colon S(n)\to\mR$ is given by $F_p(X) = \tr X + (p-2)\lambda_n(X)$.
For computational convenience we shall in this paper multiply the operator with a practical, but otherwise insignificant, scaling constant. We define
\begin{align*}
F_p(X) &:= \tfrac{1}{n+p-2}\Big(\tr X + (p-2)\lambda_n(X)\Big)&&\text{for $2\leq p<\infty$, and}\\
F_\infty(X) &:= \lambda_n(X),&&\text{the largest eigenvalue of $X$.}
\end{align*}
The normalization makes
\[F_p(X + mI) = F_p(X) + m\]
for all $p\in[2,\infty]$, $X\in S(n)$, and $m\in\mR$.
We let
\[\Theta_p := \Theta(F_p) = \{X\in S(n)\;|\; F_p(X)\leq 0\}\]
denote the sublevel set of $F_p$. It can be verified that $\Theta_p$ is a closed convex cone in $S(n)$. At $p=\infty$,
\[\Theta_\infty = \{X\in S(n)\;|\; \lambda_n(X)\leq 0\} = \{X\in S(n)\;|\; X\leq 0\}\]
is the set $S_-(n)$ of negative semidefinite matrices. When $p$ decreases, the cone $\Theta_p$ gradually opens, and eventually flattens out to the half-space
\[\Theta_2 = \{X\in S(n)\;|\; \tr X\leq 0\} = \{X\in S(n)\;|\; \langle I,X\rangle\leq 0\}.\]
See Figure \ref{fig:good}. For $2\leq p'<p\leq\infty$, one can check that
\begin{equation}
\Theta_\infty\subseteq \Theta_p\subseteq \Theta_{p'}\subseteq \Theta_2\qquad \text{and}\qquad \partial \Theta_p\cap \partial \Theta_{p'} = \{0\}.
\label{eq:nesting}
\end{equation}

As an immediate consequence of \eqref{eq:domination}, a function $u\colon\Omega\to(-\infty,\infty]$ is $p$-superharmonic whenever it is dominative $p$-superharmonic (\cite{MR4085709}). Thus, if we can find an invertible $B$ and an $X_0\in S(n)$ such that
\begin{equation}
B^\top \Theta B - \{X_0\}\subseteq \Theta_p
\label{eq:suffcond}
\end{equation}
for some $2\leq p\leq n$, then every supersolution $u$ of $F(\cH u) = 0$ is, by a change of variables and by subtracting a quadratic, $p$-superharmonic. We conclude that $u\in W^{1,q}_{loc}(\Omega)$ for every $0<q<\frac{n}{n-1}(p-1)$ by Theorem 5.18 in \cite{MR3931688}. We shall show that
\begin{equation}
\ac(B^\top \Theta B)\subseteq \Theta_p
\label{eq:suffcond2}
\end{equation}
is a strictly weaker condition than \eqref{eq:suffcond}, but still sufficient in order to ensure $u\in W^{1,q}_{loc}(\Omega)$ in the case $2<p\leq n$. 

We find it rather interesting that \eqref{eq:suffcond2} turns out to also be \emph{necessary} under some common assumptions on $F$.

\begin{theorem}[Sufficient condition]\label{thm:suff}
Let $p\in(2,n]$, $\Omega\subseteq\mR^n$ be open, and let $\Theta\subseteq S(n)$ be the sublevel set of an operator $F\colon S(n)\to\oR$.\footnote{$\Theta$ is well-defined even if $F$ takes values in the extended real line $\oR := \mR\cup\{\pm\infty\}$.}

If there is an invertible $n\times n$ matrix $B$ such that
\[\ac(B^\top \Theta B)\subseteq \Theta_p\]
then every viscosity supersolution of
\[F(\cH u) = 0\qquad\text{in $\Omega$}\]
is $W^{1,q}_{loc}(\Omega)$ for all $q$ such that
\[0 < q < \frac{n}{n-1}(p-1).\]
\end{theorem}

When $p=2$ the condition \eqref{eq:suffcond2} is necessary (under our assumptions), but our method of proof for sufficiency does no longer work. In this case we have to assume \eqref{eq:suffcond}, which can be rephrased as $\Theta$ being confined to an affine half-space with positive definite outer normal.

\begin{proposition}[Sufficient condition. $p=2$]\label{prop:suff}
If there is a $m\in\mR$ and a positive definite $n\times n$ matrix $A$ such that
\begin{equation}
\Theta\subseteq\{X\in S(n)\;|\;\langle X,A\rangle \leq m\},
\label{eq:2cond}
\end{equation}
then every viscosity supersolution of $F(\cH u) = 0$
is $W^{1,q}_{loc}$ for all $0 < q < \frac{n}{n-1}$.
\end{proposition}

\begin{figure}[h]
	\centering
	\begin{subfigure}[b]{0.45\textwidth}
		\centering
		\includegraphics{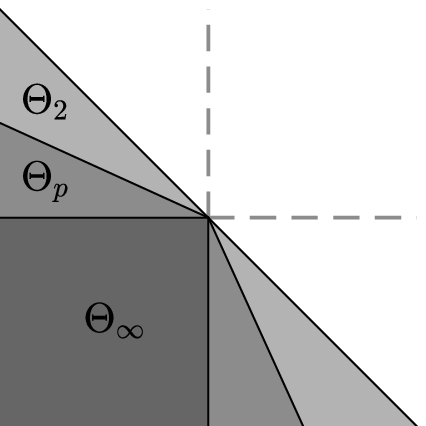}
		\caption{}
		\label{fig:good}
	\end{subfigure}
	\begin{subfigure}[b]{0.45\textwidth}
		\centering
		\includegraphics{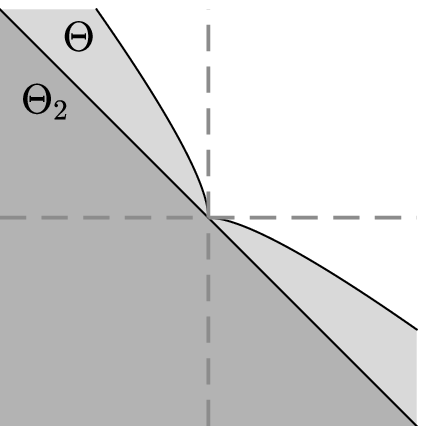}
		\caption{}
		\label{fig:bad}
	\end{subfigure}
	\caption{(a): Sublevel sets of the dominative $p$-Laplacian in a simplified model of $S(n)$.
	(b): The sublevel set $\Theta$ of equation \eqref{eq:ex} together with its asymptotic cone $\ac(\Theta) = \Theta_2$.}\label{fig:goodbad}
\end{figure}


The condition \eqref{eq:2cond} is probably not necessary as indicated by the following example. Consider the equation
\begin{equation}
\lambda_1(\cH u) + \lambda_2(\cH u) - 2\sqrt{1+\lambda_2(\cH u)} + 2 = 0
\label{eq:ex}
\end{equation}
in the unit ball $B_1$ in $\mR^2$.
One can check that radial solutions $w(x) = U(|x|)$ with $U(0) = +\infty$ are on the form $w(x) =  - \frac{1}{2}|x|^2 + 2c|x| -c^2\ln|x|$, $c\geq 1$, and are thus $W^{1,q}_{loc}(B_1)$ for all $q<n/(n-1) = 2$.
However, \eqref{eq:2cond} does not hold since the trace of the Hessian matrix of $w$ is not bounded above as $x\to 0$.
Nevertheless, $\ac(\Theta)\subseteq\Theta_2$, because if $X_k\in\Theta$ and $t_k\to\infty$ are sequences such that $X_k/t_k\to Z\in \ac(\Theta)$, then $t_k$ must be comparable to $\lambda_n(X_k)$ for large $k$ and
\[\tr Z = \lim_{k\to\infty}\frac{\tr X_k}{t_k} \leq \lim_{k\to\infty}\frac{2\sqrt{1+\lambda_n(X_k)} - 2}{t_k} = 0.\]
We conjecture that Theorem \ref{thm:suff} is valid also for $p=2$.

In order to state the necessary conditions, we establish some terminology. An operator $F\colon S(n)\to\oR$ is said to be \emph{rotationally invariant} (also called \emph{spectrally defined}) if it only depends on the eigenvalues of the argument. This is equivalent to
\[F(X) = F(Q^\top XQ)\qquad\text{for all $X\in S(n)$ and all $Q\in O(n)$,}\]
where $O(n)$ is the set of $n\times n$ orthogonal matrices. The Laplacian, the dominative $p$-Laplacian, the Pucci operators, and Monge-Amprère operators are typical examples from this class. Linear operators $F(X) = \tr(AX)$ are counterexamples provided $A$ is not a scaling of the identity matrix. However, if $A$ is positive definite, then $F$ can be \emph{made} rotationally invariant by a linear change of variables. Indeed, $X\mapsto F\big(\sqrt{A}^{-1}X\sqrt{A}^{-1})$ is the Laplacian.

\begin{definition}
An operator $F\colon S(n)\to\oR$ is \emph{essentially rotationally invariant} if there is an invertible matrix $B$ such that
\[X\mapsto F(BXB^\top )\]
is rotationally invariant.
\end{definition}

An operator is \emph{elliptic} if
\begin{equation}
X\leq Y\qquad\text{implies}\qquad F(X)\leq F(Y).
\label{eq:elldefstand}
\end{equation}
As always, $X\leq Y$ is the standard partial ordering in $S(n)$ and means $X-Y\in S_-(n)$. Ellipticity ensures that the sublevel set $\Theta$ of $F$ is a (negative) \emph{elliptic set}. That is,
\[\text{$X\leq Y$ and $Y\in\Theta$}\qquad\text{implies}\qquad X\in\Theta.\]
An equivalent statement can be made i terms of Minkovski addition,
\[\Theta + S_-(n) = \Theta.\]
In terms of the model of $S(n)$ used in Figure \ref{fig:goodbad}, the boundary of an elliptic set will appear as the graph of a nonincreasing function.
Next, if $F$ is rotationally invariant, then $\Theta$ is a \emph{rotationally invariant set}:
\[\Theta = \rot\Theta := \{Q^\top XQ\;|\; X\in\Theta,\,Q\in O(n)\}.\]
Similarly, if $F$ is essentially rotationally invariant, then $\Theta$ is essentially rotationally invariant, meaning that there is an invertible $B$ so that $B^\top \Theta B$ is a rotationally invariant set. Finally, $\Theta$ is convex whenever $F$ is convex, but also, for example, if $F$ is merely quasiconvex.

\begin{theorem}[Necessary condition]\label{thm:nec}
Let $p\in[2,\infty]$, $\Omega\subseteq\mR^n$ be open, and assume that $\Theta = \Theta(F) \subseteq S(n)$ is an elliptic, essentially rotationally invariant, and convex sublevel set of an operator $F\colon S(n)\to\oR$.

If every viscosity supersolution of
\[F(\cH u) = 0\qquad\text{in $\Omega$}\]
is $W^{1,q}_{loc}(\Omega)$ for all $q$ with
\[0 < q < \frac{n}{n-1}(p-1),\]
then there is an invertible $n\times n$ matrix $B$ such that
\begin{equation}
\ac(B^\top \Theta B)\subseteq \Theta_p.
\label{eq:nec}
\end{equation}
\end{theorem}

The proofs of Theorem \ref{thm:suff} and Proposition \ref{prop:suff} follow below. The proof of Theorem \ref{thm:nec} is postponed until Section \ref{sec:acdo}.

\begin{proof}[Proof of Theorem \ref{thm:suff}]
Let $2<p\leq n$. The \emph{associated convex body} to the dominative $p$-Laplacian is the set
\[\cK_p := \left\{A\in S(n)\;\middle|\; \tfrac{1}{n+p-2}I\leq A\leq \tfrac{p-1}{n+p-2}I,\,\tr A = 1\right\}.\]
See Section \ref{sec:risee}.
It is the unique convex and compact subset of $S(n)$ such that $F_p$ is the support function of $\cK_p$. That is,
\[F_p(X) = \max_{A\in\cK_p}\langle A,X\rangle.\]
Introduce the short-hand
\[\Theta^B := B^\top\Theta B\]
and suppose
\[\ac(\Theta^B)  \subseteq \Theta_p\]
for some invertible matrix $B$.
Let $0 < q < \frac{n}{n-1}(p-1)$, and choose $p'\in[2,p)$ such that we still have $q < \frac{n}{n-1}(p'-1)$. Now,
\[\cK_{p'} \subseteq\cK_p\]
and we claim that
\begin{equation}
\sup_{\substack{A\in\cK_{p'}\\ X\in\Theta^B}}\langle A,X\rangle
\label{eq:s}
\end{equation}
is finite.
Suppose it is not. It is $-\infty$ only if $\Theta$ is empty, but then there are no supersolutions and nothing to prove in the Theorem. There are therefore sequences $A_k\in\cK_{p'}$ and $X_k\in\Theta^B$ with $\langle A_k,X_k\rangle\to\infty$ as $k\to\infty$. By compactness, we may assume $A_k$ to converge to some $A_0\in\cK_{p'}$ and $\hat{X}_k := X_k/\|X_k\|$ to  converge to some $Z$ in the unit sphere in $S(n)$. Obviously, $\limsup_{k\to\infty}\|X_k\|=\infty$, and $\langle A_k,\hat{X}_k\rangle$ is eventually non-negative. Thus,
\[0\leq \langle A_0,Z\rangle \leq \max_{A\in\cK_{p'}}\langle A,Z\rangle \leq \max_{A\in\cK_p}\langle A,Z\rangle = F_p(Z)\leq 0\]
since $Z\in\ac(\Theta^B)\subseteq \Theta_p$. Therefore, $F_{p'}(Z) = 0 = F_p(Z)$, which leads to the contradiction $Z=0$ by \eqref{eq:nesting}.

Let $u$ be a supersolution to $F(\cH u) = 0$ in some $\Omega\subseteq\mR^n$. Define the function
\[v(x) := u\big(Bx\big) - \frac{m}{2}|x|^2\]
where $m$ is the number \eqref{eq:s}.
Suppose $\phi$ is a test function touching $v$ from below at $x_0\in\Omega^B := \{x\;|\; Bx\in\Omega\}$, and consider the test function
\[\psi(y) := \phi\big(B^{-1}y\big) + \frac{m}{2}y^\top B^{-\top}B^{-1}y.\]
Then, $\psi(y) \leq u(y)$, and at $y_0 := Bx_0$ we have $\psi(y_0) = u(y_0)$. Thus $F(\cH\psi(y_0))\leq 0$ and
\[\Theta\ni \cH\psi(y_0) = B^{-\top}\cH\phi(x_0)B^{-1} + mB^{-\top}B^{-1}.\]
Equivalently,
\[\cH\phi(x_0) + mI \in \Theta_B.\]
Now,
\begin{align*}
F_{p'}\big(\cH\phi(x_0)\big) &= F_{p'}\big(\cH\phi(x_0) + mI\big) - m\\
                     &= \max_{A\in\cK_{p'}}\langle A,\cH\phi(x_0) + mI\rangle - m\\
										 &\leq \sup_{\substack{A\in\cK_{p'}\\ X\in\Theta^B}}\langle A,X\rangle - m\\
										 &= 0,
\end{align*}
which proves that $v$ is dominative $p'$-superharmonic in $\Omega^B$.
This implies $p'$-superharmonicity by Proposition 5 in \cite{MR4085709} and $v$ is then $W^{1,q}_{loc}(\Omega^B)$ by Thm 5.18 in \cite{MR3931688}. It follows that $u(x) = v\big(B^{-1}x\big) + \frac{m}{2}x^\top B^{-\top}B^{-1}x$ is $W^{1,q}_{loc}(\Omega)$ as well.
\end{proof}

\begin{proof}[Proof of Proposition \ref{prop:suff}]
As in the above proof, a change of variables in the test functions will show that
\[v(x) := u\big(\sqrt{A}x\big) - \frac{m}{2n}|x|^2\]
is superharmonic whenever $u$ is a supersolution of $F(\cH u) = 0$.
\end{proof}

\section{Rotationally invariant sublinear elliptic operators}\label{sec:rotinvsubellop}\label{sec:risee}

An operator $G\colon S(n)\to \mR$ is \emph{sublinear} if it is positive homogenous and subadditive. That is, for $X,Y\in S(n)$ and positive numbers $c$ we have
$G(cX) = cG(X)$ and $G(X+Y)\leq G(X)+G(Y)$.
This class of operators is nothing but the family of \emph{support functions} in $S(n)$. There is thus a unique compact and convex subset $\emptyset\neq\cK = \cK(G)\subseteq S(n)$ such that
\[G(X) = \max_{A\in\cK}\langle A,X\rangle.\]
See Theorem 1.7.1 and the foregoing discussion in \cite{MR3155183}.

It can be checked that the sublinear operator is elliptic if and only if $\cK$ is a subset of $S_+(n) := \{A\in S(n)\;|\;A\geq 0\}$, and it is rotationally invariant if and only if
\[\cK = \rot\cK := \{QAQ^\top \;|\; A\in\cK,\, Q\in O(n)\}.\]
Furthermore, we label $G$ as \emph{non-totally degenerate} if $0\notin\cK$. Though not strictly necessary, this pragmatic assumption simplifies the exposition. It is a rather natural condition because $0\in\cK\;\Rightarrow\;G\geq 0$ in $S(n)$.

To each such operator we assign a number $p\in[2,\infty]$.

\begin{definition}\label{def:bca}
The \emph{body cone aperture} to a non-totally degenerate rotationally invariant sublinear elliptic operator \(G\colon S(n)\to\mathbb{R}\) with associated convex body $\cK\subseteq S(n)$ is
\[p = p(G) :=
\begin{cases}
\frac{n + \alpha - 2}{\alpha - 1},\qquad &\text{if $1<\alpha\leq n$,}\\
\infty, &\text{if $\alpha=1$,}
\end{cases}
\]
where
\begin{equation}
\alpha = \alpha(G) := \min_{A\in\cK}\frac{\tr A}{\lambda_n(A)}.
\label{eq:sca}
\end{equation}
\end{definition}

Observe that $p$ and $\alpha$ are well-defined. In fact, the trace and the largest eigenvalue $\lambda_n(A)>0$ are continuous functions of $A$, and $\cK$ is compact. Additionally,
\[1 = \frac{\lambda_n(A)}{\lambda_n(A)}\leq \frac{\tr A}{\lambda_n(A)}\leq \frac{n\lambda_n(A)}{\lambda_n(A)} = n\]
for all $A\in S_+(n)\setminus\{0\}$. The numbers are duals in the sense $(\alpha-1)(p-1) = n-1$. As $\cK\subseteq S_+(n)$, one can also note that \eqref{eq:sca} is a minimum of a ratio of the norms
\[\|X\|_1 := \sum_{i=1}^n|\lambda_i(X)|\qquad\text{and}\qquad \|X\|_\infty := \max\{-\lambda_1(X),\lambda_n(X)\}.\]
Moreover, $\alpha$ is invariant under positive scalings of the convex body. The body cone aperture is therefore -- as suggested by its name -- determined by the convex cone $\{cA\;|\; A\in\cK,\, c\geq 0\}$ in $S(n)$.

As examples, we mention the Pucci operator $F_{\lambda,\Lambda}$,
defined by the convex body
\[\cK_{\lambda,\Lambda} := \left\{A\in S(n)\;\middle|\; \lambda I\leq A\leq \Lambda I\right\},\qquad 0 < \lambda \leq \Lambda,\]
and the dominative $p$-Laplace operator
\begin{align*}
F_p(X) &= \frac{1}{n+p-2}\Big(\tr X + (p-2)\lambda_n(X)\Big), && p\in[2,\infty),\\
F_\infty(X) &= \lambda_n(X). &&
\end{align*}
Here, $F_p(X) = \max_{A\in\cK_p}\langle A,X\rangle$
where $\cK_p := \cK(F_p)$ must be the convex hull of the compact subset
\begin{equation}
\begin{aligned}
\cE_p &:= \left\{\frac{I + (p-2)\xi\xi^\top }{n+p-2}\;\middle|\; \xi\in\mathbb{S}^{n-1}\right\}, && p\in[2,\infty),\\
\cE_\infty &:= \left\{\xi\xi^\top \;\middle|\; \xi\in\mathbb{S}^{n-1}\right\}. &&
\end{aligned}
\label{eq:Epdef}
\end{equation}
A computation will reveal that
\begin{align*}
\cK_p &= \left\{A\in S(n)\;\middle|\; \tfrac{1}{n+p-2}I\leq A\leq \tfrac{p-1}{n+p-2}I,\,\tr A = 1\right\}, && p\in[2,\infty),\\
\cK_\infty &= \left\{A\in S(n)\;\middle|\; 0\leq A\leq I,\,\tr A = 1\right\},
\end{align*}
and thus,
\[\alpha(F_p) = \min_{A\in\cK_p}\frac{\tr A}{\lambda_n(A)} =
\begin{cases}
\frac{n+p-2}{p-1},\quad & p\in[2,\infty),\\
1, & p = \infty,
\end{cases}
\]
which implies $p(F_p) = p$. By way of illustration,
\[\alpha(F_{\lambda,\Lambda}) = \min_{A\in\cK_{\lambda,\Lambda}}\frac{\tr A}{\lambda_n(A)} = \min_{A\in\cK_{\lambda,\Lambda}}\frac{\lambda_1(A) + \cdots + \lambda_{n-1}(A)}{\lambda_n(A)} + 1 = \frac{(n-1)\lambda}{\Lambda} + 1\]
and $p(F_{\lambda,\Lambda}) = \Lambda/\lambda + 1$.

The dominative $p$-Laplacian holds the special position of being the \emph{minimal} operator of its class.

\begin{proposition}[Minimal operator]\label{prop:invhol}
Let \(G\colon S(n)\to\mathbb{R}\) be a non-totally degenerate rotationally invariant sublinear elliptic operator and
let \(p\in[2,\infty]\) be its body cone aperture.
Then there exists a constant \(c>0\) such that
\begin{equation}
cF_p(X) \leq G(X)\qquad\forall X\in S(n).
\label{eq:Gbound}
\end{equation}
\end{proposition}

For $p\in[2,\infty]$ we define the lower semicontinuous function $w_{n,p}\colon \mR^n\to\mR\cup\{+\infty\}$ as
\begin{equation}
w_{n,p}(x) :=
\begin{cases}
-\frac{p-1}{p-n}|x|^\frac{p-n}{p-1}, & 2\leq p\neq n,\\
-\ln|x|, & p = n,\\
-|x|, & p=\infty,
\end{cases}
\label{sublin_fundsol}
\end{equation}
with the interpretation $w_{n,p}(0) = \infty$ for $p\leq n$. It is a solution to the dominative $p$-Laplace equation $F_p(\cH w) = 0$ in $\mR^n\setminus\{0\}$ and it is a viscosity supersolution in $\mR^n$. We show next that the same is true for the equation $G(\cH w) = 0$ when $p = p(G)$.

\begin{proposition}[Existence of fundamental solution]\label{prop:exfundsol}
Let \(G\colon S(n)\to\mathbb{R}\) be a non-totally degenerate rotationally invariant sublinear elliptic operator and
let \(p\in[2,\infty]\) be its body cone aperture.
Then \(w_{n,p}\)
is a solution to the equation $G(\cH w) = 0$ in \(\mathbb{R}^n\setminus\{0\}\) and a viscosity supersolution in \(\mathbb{R}^n\).
\end{proposition}

In particular, the bound \eqref{eq:Gbound} is sharp.

The key ingredient in the proof of Proposition \ref{prop:invhol} is established in the following Lemma. Due to rotational invariance, it can be conducted in $\mR^n$ rather than in $S(n)$. The standard basis vectors are denoted by $\bfe_1,\dots,\bfe_n$, and we write $\bben := [1,\dots,1]^\top = \bfe_1 + \cdots +\bfe_n\in\mR^n$.

\begin{lemma}\label{lem:perc}
Let $p\in[2,\infty]$ and set $\bfp\in\mR^n$ to be
\begin{equation}
\bfp :=
\begin{cases}
\frac{1}{n+p-2}[1,\dots,1,p-1]^\top,\qquad & \text{if $p\in[2,\infty)$,}\\
[0,\dots,0,1]^\top = \bfe_n, & \text{if $p=\infty$.}
\end{cases}
\label{eq:bfpdef}
\end{equation}
Suppose $\bfa = [a_1,\dots,a_n]^\top\in\mR^n$ is a vector such that the sum of its elements equals the sum of the elements in $\bfp$, and such that $a_n$ is equal to the last entry of $\bfp$. i.e.,
\begin{equation}
\bben^\top\bfa = 1 = \bben^\top\bfp\qquad\text{and}\qquad \bfe_n^\top\bfa = \bfe_n^\top\bfp.
\label{eq:majcond}
\end{equation}
Then $\bfp$ is in the convex hull of the set of vectors in $\mR^n$ obtained by permuting the elements in $\bfa$. In symbols,
\[\bfp\in \con\{P\bfa\;|\; P\in\mathcal{P}(n)\}\]
where $\mathcal{P}(n)$ is the set of $n\times n$ permutation matrices.
\end{lemma}

\begin{proof}
Let $\tilde{P}\in \mathcal{P}(n-1)$ be a permutation with no cycles of order less than $n-1$. For example,
\[\tilde{P} =
\begin{bmatrix}
	0 & 0 & \cdots & 0 & 1\\
	1 & 0 & \cdots & 0 & 0\\
	0 & 1 & \cdots & 0 & 0\\
	\vdots & \vdots & \ddots & \vdots & \vdots\\
	0 & 0 & \cdots & 1 & 0
\end{bmatrix} \in\mR^{(n-1)\times(n-1)}.\]
Then $\tilde{P}^{n-1} = I_{n-1}$ and
\[\tilde{P} + \tilde{P}^2 + \cdots + \tilde{P}^{n-1} = \tilde{\bben}\tilde{\bben}^\top\]
is the $(n-1)\times(n-1)$ matrix with all ones. Here, $\tilde{\bben} := [1,\dots,1]^\top\in\mR^{n-1}$. Let
\[P := 
\begin{bmatrix}
	\tilde{P} & \bf0\\
	\bf0^\top & 1
\end{bmatrix}\in \mathcal{P}(n),\]
and write $\tilde{\bfa} := [a_1,\dots,a_{n-1}]^\top$.
Now $P^k = \big[\begin{smallmatrix}
	\tilde{P}^k & \bf0\\
	\bf0^\top & 1
\end{smallmatrix}\big]$ so,
\begin{align*}
\con\{P\bfa\;|\; P\in\mathcal{P}(n)\}
	&\ni \sum_{k=1}^{n-1}\frac{1}{n-1}P^k\bfa\\
	&= \frac{1}{n-1}
	\begin{bmatrix}
		\tilde{\bben}\tilde{\bben}^\top & \bf0\\
		\bf0^\top & n-1
	\end{bmatrix}
	\begin{bmatrix}
		\tilde{\bfa}\\ a_n
	\end{bmatrix}\\
	&= 
	\begin{bmatrix}
		\dfrac{\tilde{\bben}^\top\tilde{\bfa}}{n-1}\tilde{\bben}\\
		a_n
	\end{bmatrix}.
\end{align*}
By \eqref{eq:majcond}, this equals $\bfp$: If $p=\infty$, then $\tilde{\bben}^\top\tilde{\bfa} = 1-a_n = 0$, and when $p<\infty$,
\[\frac{\tilde{\bben}^\top\tilde{\bfa}}{n-1} = \frac{1-a_n}{n-1} = \frac{1-\frac{p-1}{n+p-2}}{n-1} = \frac{1}{n+p-2}.\]
\end{proof}

\begin{proof}[Proof of Proposition \ref{prop:invhol}]

We have
\[\alpha = \alpha(G) = \min_{A\in\cK}\frac{\tr A}{\lambda_n(A)} = \frac{\tr A'}{\lambda_n(A')}\]
for some $A'$ in the associated convex body $\cK$ of $G$. Set
\[c := \tr A'>0\]
and let $A_0 := A'/c$. Then $\tr A_0 = 1$ and the vector
\[\bfa := [\lambda_1(A_0),\dots,\lambda_n(A_0)]^\top\]
satisfies $\bben^\top\bfa = 1$. Moreover, $\lambda_n(A_0) = 1/\alpha$ and when $p = p(G)\in[2,\infty]$ is the body cone aperture of $G$, then
\[\bfe_n^\top\bfa = \frac{1}{\alpha} =
\begin{cases}
\frac{p-1}{n+p-2},\qquad & p\in[2,\infty),\\
1, & p = \infty,
\end{cases}
\]
which equals the last element of the the vector $\bfp\in\mR^n$ given by \eqref{eq:bfpdef} in Lemma \ref{lem:perc}. Thus,
\begin{equation}
\bfp\in\con\{P\bfa\;|\;P\in\mathcal{P}(n)\}.
\label{eq:pcon}
\end{equation}

By using the standard property
\[\diag(Pz) = P(\diag z)P^\top\qquad\forall z\in\mR^n,\]
of permutation matrices $P$,
we want to show that the dominative body $\cK_p$ is a subset of $\frac{1}{c}\cK$.
Of course, $\diag\colon\mR^n\to S(n)$ is the linear mapping $\diag z := \sum_{k=1}^n z_k\bfe_k\bfe_k^\top$.

Since $\bfp = [\lambda_1(E),\dots,\lambda_n(E)]^\top$ for every $E\in\cE_p$ (see formula \eqref{eq:Epdef}), we can choose $Q\in O(n)$ such that $Q^\top EQ = \diag\bfp$.
By \eqref{eq:pcon} we can write $\bfp$ as a convex combination $\sum_i \alpha_i P_i\bfa$. Thus,
\begin{align*}
Q^\top EQ &= \diag\left(\sum_i \alpha_i P_i\bfa\right)\\
      &= \sum_i \alpha_i \diag\left(P_i\bfa\right)\\
			&= \sum_i \alpha_i P_i(\diag\bfa)P_i^\top\\
			&= \frac{1}{c}\sum_i \alpha_i P_iU^\top A'UP_i^\top
\end{align*}
for some $U\in O(n)$ diagonalizing $A'\in\cK$. The orthogonal matrices $Q_i := QP_iU^\top$ then makes
\[E = \frac{1}{c}\sum_i \alpha_i Q_i A'Q_i^\top \in \frac{1}{c}\con\rot\{A'\}\subseteq \frac{1}{c}\con\rot\cK = \frac{1}{c}\cK.\]
That is, $\cE_p\subseteq \frac{1}{c}\cK$ and
\[\cK_p = \con\cE_p \subseteq \frac{1}{c}\cK\]
as well. Thus, for any $X\in S(n)$,
\[cF_p(X) = c\max_{A\in\cK_p}\tr(AX) = \max_{A\in c\cK_p}\tr(AX) \leq \max_{A\in\cK}\tr(AX) = G(X).\]
\end{proof}

\begin{proof}[Proof of Proposition \ref{prop:exfundsol}]

For $p\in[2,\infty]$ one can check that the Hessian matrix $\cH w_{n,p}\colon \mR^n\setminus\{0\}\to S(n)$ of the fundamental solution is
\[\cH w_{n,p}(x) = |x|^{-\alpha}\Big( (\alpha-1)\hx\hx^\top - (I-\hx\hx^\top) \Big),\qquad \hx := \frac{x}{|x|},\]
where $\alpha\in[1,n]$ is related to $p$ as in Definition \ref{def:bca}.
Setting
\[\Lambda_\alpha := \diag(\alpha-1,-1,\dots,-1) = \alpha\bfe_1\bfe_1^\top - I,\]
produces a diagonalization $|x|^{-\alpha}\Lambda_\alpha$ of $\cH w_{n,p}(x)$.
Since $G$ is rotational invariant and positive homogenous,
\[G(\cH w_{n,p}(x)) = |x|^{-\alpha}G(\Lambda_\alpha)\]
and will vanish independently of $x\neq 0$ if we can show that $G(\Lambda_\alpha) = 0$. Indeed, when $p$ is the body cone aperture of $G$, Proposition \ref{prop:invhol} and the fact $\tr A/\lambda_n(A) \geq \alpha$ for all $A\in\cK$  yields
\begin{align*}
0  = cF_p(\Lambda_\alpha) &\leq G(\Lambda_\alpha)\\
  &= \max_{A\in\cK}\tr(A\Lambda_\alpha)\\
	&= \max_{A\in\cK} \alpha\, \bfe_1^\top A\bfe_1 - \tr A\\
	&\leq \max_{A\in\cK} \alpha \lambda_n(A) - \tr A \leq 0,
\end{align*}
and $w_{n,p}$ is a smooth solution of $G(\cH w) = 0$ in $\mR^n\setminus\{0\}$. There are no test functions touching the fundamental solution from below at $x=0$, and $w_{n,p}$ is therefore also a viscosity supersolution in $\mR^n$.
\end{proof}

We conclude this Section with an observation regarding uniformly elliptic operators,
\[\lambda\tr A \leq F(X+A)-F(X)\leq \Lambda\tr A,\qquad \forall A\geq 0.\]
Here, $0<\lambda\leq\Lambda\in\mR$ are the ellipticity constants of $F$. The above is equivalent to
\begin{equation}
F(X+Y) - F(X) \leq F_{\lambda,\Lambda}(Y),\qquad \forall X,Y\in S(n),
\label{eq:unifdef}
\end{equation}
where $F_{\lambda,\Lambda}$ is the Pucci operator. Although uniform ellipticity is in many settings a desirable property of $F$, it has a negative impact on the question raised in this paper.
The ``more'' uniformly elliptic the equation is, the ``less'' integrable are the gradients of the supersolutions.
In fact, since the body cone aperture of $F_{\lambda,\Lambda}$ is $p := \frac{\Lambda}{\lambda}+1$,  Proposition \ref{prop:exfundsol} and \eqref{eq:unifdef} implies that $w(x) := w_{n,p}(x) + \frac{1}{2}x^\top X_0 x$, $X_0\in\Theta(F)$, is a viscosity supersolution of $F(\cH u) = 0$, which is not in $W^{1,q}_{loc}(\mR^n)$ for
\[q = \frac{n}{n-1}(p-1) = \frac{n}{n-1}\frac{\Lambda}{\lambda}.\]

\section{Associated consistent distance operators and the proof of Theorem \ref{thm:nec}}\label{sec:acdo}

Let $\Theta$ be a negative elliptic and proper subset of $S(n)$. i.e.,
\[\emptyset\neq\Theta\neq S(n)\qquad\text{and}\qquad \Theta + S_-(n) = \Theta.\]
The \emph{associated consistent distance operator} (acdo) to $\Theta$ is the function $\oF$ defined on $S(n)$ as
\begin{equation}
\oF(X) := -\sup\{t\in\mR\;|\; X + tI\in\Theta\}.
\label{eq:acdodef}
\end{equation}
In \cite{brustad2020comparison} we prove a more general version of the following.

\begin{proposition}
The acdo $\oF$ to a negative elliptic and proper subset $\Theta$
is finite and elliptic (in the standard sense \eqref{eq:elldefstand}). It is 1-Lipschitz and has the nondegeneracy
\[\oF(X+\tau I) - \oF(X) = \tau\]
for all $X\in S(n)$, $\tau \in\mR$.
Moreover, if $\Theta$ is a sublevel set of an operator $F$, then every viscosity supersolution of $F(\cH u) = 0$ is also a viscosity supersolution of $\oF(\cH u) = 0$. The opposite inclusion holds if $\Theta$ is closed.
\end{proposition}

The acdo is in fact the signed distance function
\[\oF(X) =
\begin{cases}
\dist(X,\partial\Theta),\qquad &X\notin\Theta,\\
-\dist(X,\partial\Theta), &X\in\Theta,
\end{cases}
\]
from the boundary of $\Theta$ when the distance $\dist(X,\partial\Theta) := \inf_{W\in\partial\Theta}\|X-W\|_\infty$ is measured in the infinity norm $\|X\|_\infty := \max\{-\lambda_1(X),\lambda_n(X)\}$. For $X\in S(n)$, $\oF(X)$ is the unique number such that
\begin{equation}
X - \oF(X)I\in\partial\Theta.
\label{eq:bndacdo}
\end{equation}

In addition to the ellipticity and uniform continuity, the associated consistent distance operator \eqref{eq:acdodef} can have desirable global properties that may not be present in the original operator $F$.

\begin{proposition}\label{prop:Gprop}
Let $\emptyset\neq\Theta\neq S(n)$ be an elliptic set. Then the following hold.
\begin{enumerate}[(a)]
	\item If $\Theta$ is convex, then $\oF$ is convex.
	\item If $S(n)\setminus\Theta$ is convex, then $\oF$ is concave.
		\item If $\partial\Theta$ is a (hyperplane/subspace) in $S(n)$, then $\oF$ is (affine/linear).
	\item If $\Theta$ is a cone, then $\oF$ is positively homogeneous. i.e.,
	\[\oF(cX) = c\oF(X)\qquad\forall c> 0,\,X\in S(n).\]
	\item If $\Theta$ is a convex cone, then $\oF$ is sublinear.
	\item If $\Theta$ is a rotationally invariant set, then $\oF$ is rotationally invariant.
\end{enumerate}
\end{proposition}

\begin{proof}
(a): Since the closure of a convex set is convex, we may assume that $\Theta$ is closed.
Let $X,Y\in S(n)$ and let $\gamma\in[0,1]$. Then $Z := \gamma X + (1-\gamma)Y\in\Theta$ and $X-\oF(X)I\in\partial\Theta$, $Y-\oF(Y)I\in\partial\Theta$. Thus,
\[Z - \big( \gamma \oF(X) + (1-\gamma)\oF(Y) \big) I = \gamma\big(X-\oF(X)I\big) + (1-\gamma)\big(Y-\oF(Y)I\big) \in\Theta,\]
and
\[- \oF(Z) = \sup\{t\;|\; Z+tI\in\Theta\} \geq - \big( \gamma \oF(X) + (1-\gamma)\oF(Y) \big).\]

(b): One can show that $\tilde{\Theta} := -(S(n)\setminus\Theta)$ is elliptic. Then since
\begin{align*}
X\mapsto - \oF(-X) &= \sup\{t\;|\; -X+tI\in\Theta\}\\
       &= \inf\{t\;|\; -X+tI\notin\Theta\}\\
			 &= -\sup\{t\;|\; -X-tI\notin\Theta\}\\
			 &= -\sup\{t\;|\; X+tI\in\tilde{\Theta}\}
\end{align*}
is convex by (a), it follows that $\oF$ is concave.

(c): By ellipticity, $\partial\Theta$ is necessarily a hyperplane $\partial\Theta = \{X\;|\;\langle A,X\rangle= m\}$ for some nonzero positive semidefinite matrix $A$ and $m\in\mR$. Thus by \eqref{eq:bndacdo}, $\langle A,X-\oF(X)I\rangle= m$ for all $X$ and
\[\oF(X) = \frac{1}{\tr A}\tr(AX) - \frac{m}{\tr A}.\]

(d): Since $X-\oF(X)I\in\partial\Theta$, then also $cX-c\oF(X)I\in\partial\Theta$ when $\Theta$ is a cone. That is,
\[0 = \oF\big(cX - c\oF(X)I\big) = \oF(cX) - c\oF(X).\]

(e): This is immediate from (a) and (d).

(f): $\oF$ is rotationally invariant since
\begin{align*}
\left\{t\;|\; Q^\top XQ + tI\in\Theta\right\}
	&= \left\{t\;|\; Q^\top (X + tI)Q\in\Theta\right\}\\
	&= \left\{t\;|\; X + tI\in Q\Theta Q^\top \right\}\\
	&= \left\{t\;|\; X + tI\in\Theta\right\}.
\end{align*}
\end{proof}

A final lemma is needed before we can prove the necessity of the condition $\ac(\Theta^B)\subset\Theta_p$.

\begin{lemma}\label{lem:final}
Let $\Theta\subseteq S(n)$ be a proper negative elliptic set, and let $B\in\mR^{n\times n}$ be invertible. Then
$\Theta^B := B^\top\Theta B$ and $\ac(\Theta)$ are again proper negative elliptic sets. If $\Theta$ is rotationally invariant, then so is $\ac(\Theta)$.
\end{lemma}

\begin{proof}
Let $X,Y\in S(n)$ with $X\leq Y$ and $Y\in\Theta^B$. Then $Y = B^\top\tilde{Y}B$ for some $\tilde{Y}\in\Theta$. Thus,
\[\tilde{X} := B^{-\top}X B^{-1} \leq B^{-\top}Y B^{-1} = \tilde{Y}\]
and $\tilde{X}\in \Theta$ since $\Theta$ is an elliptic set. It follows that $X = B^\top\tilde{X}B\in \Theta^B$. The properness is clear.

Assume now that $X\leq Y\in\ac(\Theta)$. Let $t_k\to\infty$ and $Y_k\in\Theta$ be such that $Y_k/t_k\to Y$. Since $\Theta$ is elliptic and $t_k(X-Y)\leq 0$, we have $Y_k + t_k(X-Y)\in\Theta$ for each $k$. Thus,
\[X = \lim_{k\to\infty}\frac{Y_k + t_k(X-Y)}{t_k}\in\ac(\Theta).\]

The asymptotic cone is nonempty. In fact, $S_-(n)\subseteq\ac(\Theta)$ because if $X\leq 0$ and $Y\in\Theta$, then $Y + kX\in\Theta$ for all $k = 1,2,\dots$ and
\[X = \lim_{k\to\infty}\frac{Y + kX}{k}\in\ac(\Theta).\]
On the other hand, $\inter S_+(n)\subseteq ( S(n)\setminus\ac(\Theta) )$. Because if $Z\in\ac(\Theta)$ and say, $Z\geq \epsilon I$, then there are $X_k\in\Theta$ and $t_k\to\infty$ so that $X_k/t_k \geq (\epsilon/2)I$ for all sufficiently large $k$. Thus, eventually
\[X_k \geq \frac{\epsilon t_k}{2}I \geq Y\]
for any $Y\notin\Theta$. A contradiction.

Finally, for $Z\in\ac(\Theta)$ and $Q\in O(n)$ let $t_k\to\infty$ and $X_k\in\Theta$ be sequences so that $X_k/t_k\to Z$. Since $\Theta$ is rotationally invariant, $Q^\top X_kQ\in\Theta$ for each $k$, and since $Q^\top X_kQ/t_k\to Q^\top ZQ$ it follows that $Q^\top ZQ\in\ac(\Theta)$. Therefore,
\[\rot\ac(\Theta) := \{Q^\top ZQ\;|\; Z\in\ac(\Theta),\,Q\in O(n)\} = \ac(\Theta).\]
\end{proof}

\begin{proof}[Proof of Theorem \ref{thm:nec}]

The asymptotic cone $\ac(\Theta^B) = \ac(B^\top \Theta B)$ is easily seen to be a closed cone in $S(n)$.
As $\Theta$ is assumed to be convex, $\Theta^B$ is convex and so is $\ac(\Theta^B)$.
See Section 2.1 and 2.2, and in particular, Proposition 2.1.5 in \cite{MR1931309}.
Moreover, we are there given the equivalent formulations
\begin{align*}
\ac(\Theta^B) &= \{Z\in S(n)\;|\; \text{$X+tZ\in\clos\Theta^B$ for all $X\in\Theta^B$ and all $t\geq 0$}\}\\
          &= \{Z\in S(n)\;|\; \text{$X+tZ\in\clos\Theta^B$ for some $X\in\Theta^B$ and all $t\geq 0$}\}.
\end{align*}
Here, $\clos\Theta^B$ denotes the closure of $\Theta^B$. For convex sets the above is also called the \emph{recession cone}.

The available alternative formula for $\ac(\Theta^B)$ implies $\ac(\Theta^B) + \{X\}\subseteq \clos\Theta^B$ for all $X\in\Theta^B$. In particular, since $\Theta^B$ is elliptic by Lemma \ref{lem:final} above and since $Y\in \clos\Theta^B\;\Rightarrow\; Y-\epsilon I\in \inter \Theta^B$ for all $\epsilon>0$ by Lemma 3.1 (2) in \cite{brustad2020comparison}, we can pick a large enough $m\in\mR$ so that
\begin{equation}
\ac(\Theta^B) - m\{I\}\subseteq \Theta^B.
\label{eq:incl}
\end{equation}

We now choose $B$ so that $\Theta^B$ is rotationally invariant and let $G$ denote the associated consistent distance operator to $\ac(\Theta^B)$. That is,
\[G(X) := -\sup\{t\;|\; X+tI\in\ac(\Theta^B)\}.\]
By Lemma \ref{lem:final} and the discussion above, $\ac(\Theta^B)$ is a proper and rotationally invariant elliptic convex cone. Therefore, by Proposition \ref{prop:Gprop} (e) and (f), $G\colon S(n)\to\mR$ is a rotationally invariant sublinear elliptic operator.
As $\ac(\Theta^B)$ is closed, we also note that
\begin{equation}
G(Z)\leq 0\qquad\iff\qquad Z\in\ac(\Theta^B).
\label{eq:G}
\end{equation}

Let $p\in[2,\infty]$. We assume the negation of \eqref{eq:nec} and aim to show that $F(\cH u) = 0$ has a supersolution not in $W^{1,q}_{loc}$ for some 
\[0<q< \frac{n}{n-1}(p-1),\]
or some $0<q<\infty$ in the case $p=\infty$.

Accordingly, as $\ac(\Theta^B)$ is \emph{not} a subset of $\Theta_p$, there is a $Z\in\ac(\Theta^B)$ with $F_p(Z) > 0$.
Let $g := p(G)\in[2,\infty]$ be the body cone aperture of $G$, and assume for the sake of contradiction that $g\geq p$.
Then $\Theta_{g}\subseteq \Theta_p$ so $F_{g}(Z)>0$, and by \eqref{eq:G} and Proposition \ref{prop:invhol},
\[0 < cF_{g}(Z) \leq G(Z) \leq 0.\]
Thus, $2\leq g< p$.
Also, for $x\neq 0$, $G(\cH w_{n,g}(x)) = 0$ by Proposition \ref{prop:exfundsol}, and $\cH w_{n,g}(x)\in \ac(\Theta^B)$, again by \eqref{eq:G}.
Let $x_0\in\Omega^B$ and let $m\in\mR$ be the constant from \eqref{eq:incl}. Define the lower semicontinuous function $v\colon\Omega^B\to(-\infty,\infty]$ as
\[v(x) := w_{n,g}(x-x_0) - \frac{m}{2}|x|^2.\]
This function is not $W^{1,q}_{loc}(\Omega^B)$ for
\[q := \frac{n}{n-1}(g - 1) < \frac{n}{n-1}(p - 1).\]
There are no test functions touching from below at $x_0$. In $\Omega^B\setminus\{x_0\}$, $v$ is smooth with Hessian matrix
\[\cH v(x) = \cH w_{n,g}(x-x_0) - mI\subseteq \ac(\Theta^B) - m\{I\}\subseteq \Theta^B\]
by \eqref{eq:incl}. The change of variables $u(x) := v\big(B^{-1}x\big)$ produces a supersolution $u\notin W^{1,q}_{loc}(\Omega)$ as
\[\cH u(x) = B^{-\top}\cH v\big(B^{-1}x\big)B^{-1}\subseteq \Theta\]
and thus $F\big(\cH u(x)\big)\leq 0$.
\end{proof}

There is probably some room for improvement in Theorem \ref{thm:nec}. In particular, it should be possible to relax the convexity assumption on $\Theta$.


\paragraph{Acknowledgments:}
The problem addressed in this paper was suggested to me by Professor P. Lindqvist.

The results in Section \ref{sec:risee} are copied from the unpublished part \cite{brustad2018sublinear} of my thesis.


\bibliographystyle{alpha}
\bibliography{C:/Users/Karl_K/Documents/PhD/references}

\end{document}